\documentclass[11pt]{amsart}
\usepackage{tikz-cd}
\usepackage{verbatim}
\usepackage[colorinlistoftodos]{todonotes}
\usepackage{mathtools}
\usepackage[bookmarks,colorlinks,breaklinks]{hyperref}  
\usepackage[top=2cm, bottom=4.5cm, left=3cm, right=3cm]{geometry}
\hypersetup{linkcolor=blue,citecolor=blue,filecolor=blue,urlcolor=blue} 
\newtheorem{thm}{Theorem}[section]

\newtheorem{lemma}[thm]{Lemma}

\newtheorem{defn}[thm]{Definition}

\newtheorem{notation}[thm]{Notation}
\theoremstyle{plain}
\numberwithin{equation}{subsection}

\theoremstyle{remark}
\newtheorem{rem}[thm]{Remark}

\newcommand{\N}{{\mathbb N}}

\newcommand{\Q}{{\mathbb Q}}

\newcommand{\C}{\mathbb C}

\newcommand{\Z}{{\mathbb Z}}

\newcommand{\Fq}{\mathbb F_q}

\title[The isotrivial case in the Mordell-Lang conjecture]{The isotrivial case in the Mordell-Lang conjecture for semiabelian varieties defined over  fields of positive characteristic}

\author{Dragos Ghioca}
\address{Department of Mathematics \\ University of British Columbia \\ 1984 Mathematics Road \\ Canada V6T 1Z2}
\email{dghioca@math.ubc.ca}

\subjclass[2020]{Primary: 11G10; Secondary: 14G17}
\keywords{semiabelian varieties; finite fields; Mordell-Lang conjecture}

\begin{document}

\begin{abstract}
Let $G$ be a semiabelian variety defined over a finite subfield of an algebraically closed  field $K$ of prime characteristic. 
We  describe the intersection of a subvariety $X$ of $G$ with a finitely generated subgroup of $G(K)$.  
\end{abstract}

\maketitle


\section{Introduction}

The classical Mordell-Lang conjecture for semiabelian varieties $G$ defined over an algebraically closed field $K$ of characteristic $0$ (now a theorem due to Laurent \cite{Laurent} for algebraic  tori, to  Faltings \cite{Faltings} for abelian varieties, and to Vojta \cite{Vojta} for the general case of semiabelian varieties) says that the intersection of a subvariety $X$ of $G$ with a finitely generated subgroup $\Gamma$ of $G(K)$ is a finite union of cosets of subgroups of $\Gamma$. The statement in positive characteristic (i.e., when $K$ is an algebraically closed field of characteristic $p$) fails; see Hrushovski \cite{Hrushovski}. When $G$ is defined over a finite subfield $\Fq$ of $K$, Moosa and Scanlon \cite{F-sets} described the aforementioned intersection $X(K)\cap\Gamma$ under the additional assumption that $\Gamma$ is mapped into itself by a power of the Frobenius endomorphism $F$ of $G$ corresponding to $\Fq$, i.e., $\Gamma$ is a finitely generated $\Z[F^\ell]$-submodule of $G(K)$ (for a suitable positive integer $\ell$). The precise description of the intersection $X(K)\cap\Gamma$ is a finite union of \emph{$F$-sets} (see Definition~\ref{def:F0}), as given in Theorem~\ref{thm:R-T}.  In \cite{previous}, it was shown that if $\Gamma$ is not a $\Z[F^\ell]$-submodule (for any $\ell\in\N$), then the structure of the intersection $X(K)\cap\Gamma$ can be quite wild (see \cite[Section~2]{previous}); in particular, \cite[Examples~2.1,~2.2,~2.3]{previous} show that the intersection $X(K)\cap\Gamma$ is no longer a finite union of $F$-sets (as opposed to what the author claimed erroneously in an earlier paper~\cite{TAMS}). Furthermore, a geometric description of the intersection $X(K)\cap\Gamma$ was proven in \cite{previous}, even without assuming that $G$ is defined over a finite subfield of $K$, but with the disadvantage that this description is not intrinsic to the subgroup $\Gamma$ and instead it relies on the geometry of $G$.  In the present paper, we obtain an explicit description (see  Theorem~\ref{thm:main}) of the intersection $X(K)\cap\Gamma$, in the spirit of the original description of Moosa-Scanlon \cite{F-sets}.


\subsection{The intersection of a subvariety of a semiabelian variety defined over a finite field with a finitely generated subgroup invariant under a power of the Frobenius endomorphism}

\begin{notation}
\label{not:semi}
From this point on, we fix a semiabelian variety $G$ defined over a finite subfield $\Fq$ of an algebraically closed field $K$. We let $F$ be the Frobenius endomorphism of $G$ corresponding to $\Fq$.
\end{notation}

In order to state the result of Moosa-Scanlon \cite{F-sets}, we introduce the notion of $F$-sets.

\begin{defn}
\label{def:F0}
A \emph{groupless $F$-set} is any subset of $G(K)$ of the form:
\begin{equation}
\label{eq:F-orbits}
\left\{\alpha_0+\sum_{i=1}^r F^{k_in_i}(\alpha_i)\colon n_i\in\mathbb{N}\right\},
\end{equation}
for given integers $r\ge 0$ and $k_i\ge 1$, and given points $\alpha_0,\alpha_1,\dots, \alpha_r\in G(K)$. For any finitely generated subgroup $\Gamma\subset G(K)$,  we define a groupless $F$-set in $\Gamma$ (based in $G(K)$) as a groupless $F$-set contained in $\Gamma$. Also, an $F$-set in $\Gamma$ (based in $G(K)$) is any set of the form $S+B$, where $S$ is a groupless $F$-set in $\Gamma$ and $B$ is a subgroup of $\Gamma$ (as always, for any two subsets $A$ and $B$ of a given group, we have that $A+B$ is the set of all elements $a+b$ where $a\in A$ and $b\in B$). 
\end{defn}

Next, we state the main result of Moosa-Scanlon \cite{F-sets}.

\begin{thm}[Moosa-Scanlon \cite{F-sets}]
\label{thm:R-T}
Let $G$, $K$, $\Fq$ and $F$ be as in Notation~\ref{not:semi}. Let $X\subseteq G$ be a subvariety defined over $K$ and let $\Gamma\subset G(K)$ be a finitely generated subgroup with the property that there exists $\ell\in\N$ such that $F^\ell(\Gamma)\subseteq \Gamma$. Then $X(K)\cap\Gamma$ is a finite union of $F$-sets in $\Gamma$. 
\end{thm}


\subsection{The intersection of a subvariety of a semiabelian variety with an arbitrary finitely generated subgroup}

In order to state our main result regarding the intersection of a subvariety of  $G$ with an arbitrary finitely generated subgroup of $G$, we introduce a general arithmetic structure associated to any abelian group (see Definition~\ref{def:F1}). But first we define a general class of linear recurrence sequences (see Definition~\ref{def:F-1}). We recall that for a linear recurrence equation $\{a_n\}_{n\ge 1}$ given by the recurrence relation:
$$a_{n+m}=\sum_{i=0}^{m-1}c_ia_{n+i},$$
the characteristic equation is $x^m-c_{m-1}x^{m-1}-\cdots - c_0=0$; for more details, see \cite[Section~2.3]{CGSZ}. 

\begin{defn}
\label{def:S-closed}
A subset $S\subseteq \C^*$ is called \emph{powers-closed} if for any $r\in S$ and any non-negative integer $n$, we have that $r^n\in S$.
\end{defn}

\begin{defn}
\label{def:F-1}
Let $S\subseteq \C^*$ be a powers-closed set. A linear recurrence sequence $\{a_n\}_{n\ge 1}$ of integers is called an $S$-arithmetic sequence if the characteristic equation for the linear recurrence sequence $\{a_n\}_{n\ge 1}$ has distinct roots, all contained in $S$.
\end{defn}

So, with the notation as in Definition~\ref{def:F-1}, for an $S$-arithmetic sequence $\{a_n\}_{n\ge 1}$, there exist distinct numbers $r_1,\dots, r_m\in S$ and there exist complex numbers $d_1,\dots, d_m$ such that 
\begin{equation}
\label{eq:gen-form}
a_n=\sum_{i=1}^m d_ir_i^n\text{ for each }n\ge 1.
\end{equation}
Next, we generalize the notion of $S$-arithmetic sequences inside an arbitrary abelian group.

\begin{defn}
\label{def:F1}
Let $(\Gamma_0,+)$ be an abelian group, and let $S\subseteq \C^*$ be a powers-closed set. Given an integer $r\ge 1$, elements $P_1,\dots, P_r\in \Gamma_0$ along with finitely many $S$-arithmetic sequences $$\left\{a^{(1)}_n\right\}_{n\ge 1}\text{, }\left\{a^{(2)}_n\right\}_{n\ge 1},\cdots\text{, }\left\{a^{(r)}_n\right\}_{n\ge 1},$$ 
we define an $S$-arithmetic groupless set $\mathcal{U}$ as a set of the following form: 
\begin{equation}
\label{eq:Laurent-rem}
\mathcal{U}:=\left\{\sum_{i=1}^r a^{(i)}_{n_i}\cdot P_i\colon n_1,\cdots, n_r\ge 1\right\}.
\end{equation}
Given some finitely generated subgroup $\Gamma\subseteq \Gamma_0$, we say that $\mathcal{U}$ is an $S$-arithmetic groupless set in $\Gamma$ (based in $\Gamma_0$) if $\mathcal{U}\subseteq \Gamma$. 
Furthermore, an $S$-arithmetic set in $\Gamma$ (based in $\Gamma_0$) is defined as a set of the form $\mathcal{U}+B$, where $\mathcal{U}$ is an $S$-arithmetic groupless set in $\Gamma$, while $B$ is a subgroup of $\Gamma$.
\end{defn}

The notion of $S$-arithmetic sets is inspired by the definition of $F$-sets, as shown by the following notation. 

\begin{notation}
\label{def:still F-set}
With $G$, $K$, $\Fq$ and $F$ as in Notation~\ref{not:semi}, then  inside the endomorphism ring ${\rm End}(G)$, we have that $F$ is integral over $\Z$, i.e., there exists $m\ge 1$ along with integers $c_0,\dots, c_{m-1}$ (where $c_0\ne 0$) such that
\begin{equation}
\label{eq:F-rec}
F^m=\sum_{i=0}^{m-1}c_i F^i\text{ in }{\rm End}(G).
\end{equation}
Furthermore, the equation 
\begin{equation}
\label{eq:the_equation}
x^m-c_{m-1}x^{m-1}-\cdots c_1x-c_0=0
\end{equation} 
has distinct complex roots (for more details, see \cite[Section~2.1]{CGSZ}). We let  $S_F$ be the subset of $\C^*$ consisting of all complex numbers of the form $r^m$ for integers $m\ge 0$, where $r$ varies among the roots of the equation \eqref{eq:the_equation}.
\end{notation}

\begin{rem}
\label{rem:still F-set} 
Let $S_F$ be as in Notation~\ref{def:still F-set}. Then using equation~\eqref{eq:F-rec} (see also \cite[Section~3]{CGSZ}), we obtain that for any finitely generated subgroup $\Gamma\subset G(K)$, a groupless $F$-set in $\Gamma$ (based in $G(K)$) is also an $S_F$-arithmetic groupless set in $\Gamma$ (based in $G(K)$). Indeed, there exist $m$ linear recurrence sequences in $\mathbb{Z}$: $\{a^{(1)}_n\}_{n\in\mathbb{N}},\cdots, \{a^{(m)}_n\}_{n\in\mathbb{N}}$, each one of them satisfying the recurrence relation
$$a^{(i)}_{n+m}=\sum_{k=0}^{m-1} c_k\cdot a^{(i)}_{n+k}\text{ for }n\ge 1,$$
such that for any point $P\in G(K)$, we have
$$F^n(P)=\sum_{i=0}^{m-1}a^{(i+1)}_n\cdot F^i(P).$$
Therefore, an $F$-set in $\Gamma$ is an $S_F$-arithmetic set in $\Gamma$. On the other hand, there are many more $S_F$-arithmetic (groupless) sets in $\Gamma$, which are not (groupless) $F$-sets in $\Gamma$. Furthermore,  
\cite[Example~2.3]{previous} shows that arbitrary $S_F$-arithmetic sets may appear in the intersection of a subvariety $X\subset G$ with some finitely generated subgroup $\Gamma\subset G(K)$; in Theorem~\ref{thm:main}, we prove that \emph{always} $X\cap \Gamma$ must be an $S_F$-arithmetic set in $\Gamma$.
\end{rem}
We prove the following main result.
\begin{thm}
\label{thm:main}
Let $G$, $K$, $\Fq$ and $F$ be as in Notation~\ref{not:semi}, and let $S_F$ be as in Notation~\ref{def:still F-set}. Let $X\subseteq G$ be a subvariety defined over $K$ and let $\Gamma\subset G(K)$ be a finitely generated subgroup. Then the intersection $X(K)\cap\Gamma$ is a  union of finitely many $S_F$-arithmetic sets in $\Gamma$ (based in $G(K)$). 
\end{thm}


\subsection{Plan for our paper}

We prove Theorem~\ref{thm:main} as a consequence of a general statement regarding $S$-arithmetic sets in an abelian group.

\begin{thm}
\label{thm:G}
Let $(\Gamma_0,+)$ be an abelian group, let $\tilde{\Gamma}$ and $\Gamma$ be  finitely generated subgroups of $\Gamma_0$, and let $S\subseteq \C^*$ be a powers-closed set. Then the intersection of an $S$-arithmetic set in $\tilde{\Gamma}$ (based in $\Gamma_0$) with $\Gamma$  is a finite union of $S$-arithmetic sets in $\Gamma$ (based in $\Gamma_0$).
\end{thm}
In Remark~\ref{rem:Laurent}, we show an example that if one were to consider an extension of $S$-arithmetic sets (see Definition~\ref{def:F1}) in an abelian group which allows for linear recurrence sequences in equation~\eqref{eq:Laurent-rem} with characteristic roots of higher multiplicity, then the corresponding problem from Theorem~\ref{thm:G} leads to some deep unknown questions from classical number theory (such as determining the positive integers whose squares have a given number of nonzero digits when written in base-$2$).

We prove Theorem~\ref{thm:G} in Section~\ref{sec:proof}; then  Theorem~\ref{thm:main} follows as an easy corollary of Theorem~\ref{thm:G} (see Section~\ref{sec:proofs}).


\section{Proof of Theorem~\ref{thm:G}}
\label{sec:proof}

We prove Theorem~\ref{thm:G} over the next several subsections of Section~\ref{sec:proof}; so, throughout the entire Section~\ref{sec:proof}, we work under the hypotheses and notation from Theorem~\ref{thm:G}. 

We start with a very simple lemma, which is used repeatedly in our proofs.
\begin{lemma}
\label{lem:still_simple}
Let $S\subseteq \C$ be a powers-closed set and let $\{a_n\}_{n\in\mathbb{N}}$ be an $S$-arithmetic sequence. 
\begin{itemize}
\item[(i)] Then for each $u\in\mathbb{N}\cup\{0\}$ and $v\in\mathbb{N}$, the sequence $\{a_{un+v}\}_{n\in\mathbb{N}}$ is also an $S$-arithmetic sequence.
\item[(ii)] Then for each $u,v\in\Z$, the sequence $\{u\cdot a_n+v\}_{n\in\mathbb{N}}$ is also an $S$-arithmetic sequence.
\end{itemize}
\end{lemma}

\begin{proof}
The result follows immediately using the general form~\eqref{eq:gen-form} of an element in a linear recurrence sequence whose characteristic roots are non-repeated elements of $S$. Indeed, let $r_1,\dots, r_m$ be the characteristic roots of $\{a_n\}_{n\in\mathbb{N}}$.

For part~(i), note that the characteristic roots of the linear recurrence sequence $\{a_{un+v}\}$ are $r_i^u$; since $S$ is powers-closed, then each $r_i^u$ is also in $S$. 

For part~(ii), note that adding a nonzero constant to a linear recurrence sequence leads to another linear recurrence sequence whose characteristic roots are distinct; they form the set $\{r_1,\dots, r_m\}\cup\{1\}$ (also, note that $1\in S$ because $S$ is a powers-closed set).
\end{proof}

We proceed to proving Theorem~\ref{thm:G}; first, we will sketch the plan for our proof in Section~\ref{subsec:plan}. 

\subsection{Plan for our proof of Theorem~\ref{thm:G}}
\label{subsec:plan}

The first step in our proof is to make a couple useful reductions:
\begin{itemize}
\item in Section~\ref{subsec:fin.gen}, we show that it suffices to prove Theorem~\ref{thm:G} when $\Gamma_0=\tilde{\Gamma}$ is a finitely generated group. 
\item in Section~\ref{subsec:torsion-free}, we show that it suffices to assume $\Gamma_0$ is torsion-free.
\end{itemize}
So, having reduced the proof of Theorem~\ref{thm:G} to the case the ambient group $\Gamma_0$ is a torsion-free, finitely generated group, the next step (established in Section~\ref{subsec:linear_algebra}) is to reformulate our problem as a linear algebra question. Using the analysis from Sections~\ref{subsec:C}~and~\ref{subsec:L}, we conclude our proof of Theorem~\ref{thm:G} in Section~\ref{subsec:conclusion}.


\subsection{It suffices to prove Theorem~\ref{thm:G} assuming  $\Gamma_0=\tilde{\Gamma}$ is finitely generated}
\label{subsec:fin.gen}

Indeed, we have an $S$-arithmetic groupless set $\mathcal{U}$ (contained in $\tilde{\Gamma}$) consisting of all elements of $\Gamma_0$ of the form
\begin{equation}
\label{eq:P}
\sum_{i=1}^m a^{(i)}_{n_i}\cdot P_i\text{, as we vary }n_i\in\N,
\end{equation}
for some $m\in\N$, for some given elements $P_i\in \Gamma_0$, and for some given $S$-arithmetic sequences  $\{a^{(1)}_n\}_{n\ge 1},\cdots, \{a^{(m)}_n\}_{n\ge 1}$. Then given a subgroup $H\subseteq \tilde{\Gamma}$, our goal is to show that the $S$-arithmetic set  
\begin{equation}
\label{eq:F}
\mathcal{F}:=\mathcal{U}+H 
\end{equation}
intersects $\Gamma$ in a finite union of $S$-arithmetic sets (in $\Gamma$). At the expense of replacing $\tilde{\Gamma}$ with a larger subgroup of $\Gamma_0$ (but still finitely generated), we may assume both that $\Gamma\subseteq \tilde{\Gamma}$ and that each $P_i\in \tilde{\Gamma}$ for $i=1,\dots, m$. Finally, without loss of generality, we may assume $\Gamma_0=\tilde{\Gamma}$ is finitely generated.


\subsection{It suffices to prove Theorem~\ref{thm:G} assuming $\Gamma_0$ is torsion-free}
\label{subsec:torsion-free}

We already reduced proving Theorem~\ref{thm:G} to the special case $\Gamma_0=\tilde{\Gamma}$ is finitely generated. Therefore, since $(\Gamma_0,+)$ is an abelian group, then $\Gamma_0$ is the direct product of a (finitely generated) free subgroup $\Gamma_{0,{\rm free}}$ with a finite torsion subgroup $\Gamma_{0, {\rm tor}}$. Hence both $H$ and $\Gamma$ (being subgroups of $\Gamma_0$) are finite unions of cosets of subgroups of $\Gamma_{0, {\rm free}}$, i.e., 
\begin{equation}
\label{eq:H-tor}
H=\bigcup_{i=1}^k \left(h_i+H_{\rm free}\right)\text{ and }\Gamma=\bigcup_{j=1}^\ell \left(\gamma_{j}+\Gamma_{\rm free}\right),
\end{equation}
for some $k,\ell\in\mathbb{N}$, some elements $h_{i}$ and $\gamma_{j}$ in $\Gamma_{0}$ and some subgroups $H_{\rm free}$ and $\Gamma_{\rm free}$ of $\Gamma_{0, {\rm free}}$. So, it suffices to prove that for any given $i_0\in\{1,\dots, k\}$ and $j_0\in\{1,\dots, \ell\}$, the intersection
\begin{equation}
\label{eq:i-j}
\left(\mathcal{U}+\left(h_{i_0}+H_{\rm free}\right)\right)\cap \left(\gamma_{j_0}+\Gamma_{\rm free}\right)
\end{equation}
is a finite union of $S$-arithmetic sets in $\Gamma$. The following observation will be used repeatedly in our proof.

\begin{rem}
\label{rem:repeatedly}
Since a constant sequence (in $\Z$) is itself a linear recurrence sequence whose characteristic root is $1$ (this is covered by the case $u=0$ in  Lemma~\ref{lem:still_simple}, either parts~(i)~or~(ii)), we note that for any $S$-arithmetic groupless set $\mathcal{G}$ and for any $\delta\in\Gamma_0$, we have that also $\delta+\mathcal{G}$ is an $S$-arithmetic groupless set (note also that $S$ contains $1$). Furthermore, if $\mathcal{G}$ is an $S$-arithmetic set, then also $\delta+\mathcal{G}$ is an $S$-arithmetic set.
\end{rem}

We re-write the intersection from \eqref{eq:i-j}  as
\begin{equation}
\label{eq:rewrite i-j}
\gamma_{j_0}+\left(\left((-\gamma_{j_0}+h_{i_0})+\mathcal{U}+H_{\rm free}\right)\cap \Gamma_{\rm free}\right). 
\end{equation}
In light of Remark~\ref{rem:repeatedly}, it suffices to prove that 
\begin{equation}
\label{eq:rewrite i-j 2}
\left((-\gamma_{j_0}+h_{i_0})+\mathcal{U}+H_{\rm free}\right)\cap \Gamma_{\rm free}\text{ is an $S$-arithmetic set in $\Gamma$.} 
\end{equation}

Now, for $i\in\{1,\dots, m\}$, we write each $P_i$ (see \eqref{eq:P}) as $P_{i, {\rm tor}}+P_{i, {\rm free}}$ with $P_{i, {\rm tor}}\in \Gamma_{0, {\rm tor}}$ and $P_{i, {\rm free}}\in \Gamma_{0,\rm free}$. Also, we write $-\gamma_{j_0}+h_{i_0}=\eta_{\rm tor}+\eta_{\rm free}$ for some $\eta_{\rm tor}\in \Gamma_{0,{\rm tor}}$ and $\eta_{\rm free}\in \Gamma_{0, {\rm free}}$.

We let $M:=|\Gamma_{0, {\rm tor}}|$.  Because each linear recurrence sequence of integers is preperiodic modulo any given integer (and  thus, in particular, preperiodic modulo $M$), we obtain that the set of tuples $(n_1,\dots, n_m)$ of positive integers  satisfying
\begin{equation}
\label{eq:same-tor}
\eta_{\rm tor}+\sum_{i=1}^m a^{(i)}_{n_i}\cdot P_{i, \rm tor}=0\text{ in }\Gamma_0
\end{equation}
is a finite union of sets of the form
$$\left\{(k_1n_1+\ell_1, k_2n_2+\ell_2,\dots, k_mn_m+\ell_m)\colon \text{ for arbitrary } n_1,\dots, n_m\ge 1\right\}$$
for some given $2m$-tuples of integers:
\begin{equation}
\label{eq:tuples}
(k_1,\ell_1,k_2,\ell_2,\cdots, k_m,\ell_m)\text{ with }k_i\ge 0\text{ and }\ell_i\ge 1.
\end{equation}  
Therefore, at the expense of replacing each linear recurrence sequence $\left\{a^{(i)}_n\right\}_{n\ge 1}$ with $\left\{a^{(i)}_{k_in+\ell_i}\right\}_{n\ge 1}$ (which is still an $S$-arithmetic sequence, as shown in Lemma~\ref{lem:still_simple}~(i)), we reduce our problem to proving that the intersection with $\Gamma_{\rm free}$ of  the $S$-arithmetic set 
\begin{equation}
\label{eq:F-free}
\mathcal{F}_1:=\mathcal{S}_1+H_{\rm free},
\end{equation}
where $\mathcal{S}_1$ is the $S$-arithmetic groupless set given by
\begin{equation}
\label{eq:S-free}
\mathcal{S}_1:=\left\{\eta_{\rm free}+\sum_{i=1}^m a^{(i)}_{n_i}\cdot P_{i,{\rm free}}\colon n_1,\dots,n_m\ge 1\right\}
\end{equation}
is a finite union of $S$-arithmetic sets in $\Gamma$.

Therefore, from now on, we work under the assumption that $\Gamma_0$ is torsion-free (see also equations~\eqref{eq:F-free}~and~\eqref{eq:S-free}).


\subsection{Reduction to a linear algebra question}
\label{subsec:linear_algebra}

So, we work under the assumption that $\Gamma_0$ is a finitely generated, free abelian group; hence it is isomorphic to $\mathbb{Z}^r$ and we let $Q_1,\dots, Q_r\in\Gamma_0$ be some fixed generators for $\Gamma_0$. 

We have a subgroup $\Gamma\subseteq \Gamma_0$ and also, we have an $S$-arithmetic set $\mathcal{F}\subseteq \Gamma_0$. Furthermore, $\mathcal{F}=\mathcal{U}+H$ for an $S$-arithmetic groupless set $\mathcal{U}$ and for a subgroup $H\subseteq \Gamma_0$. Since $H$ is a subgroup of $\Gamma_0$, then it is also a finitely generated, free abelian group; thus, we let $\{R_1,\dots, R_s\}$ be a given $\Z$-basis for $H$ (where $s\le r$). For each $i\in\{1,\dots, s\}$, we write $R_i$ in terms of the basis $\{Q_1,\dots, Q_r\}$ of $\Gamma_0$ as
\begin{equation}
\label{eq:R}
R_i:=\sum_{j=1}^r b_{i,j}Q_j\text{ for some }b_{i,j}\in\Z.
\end{equation}
The $S$-arithmetic groupless set $\mathcal{U}$ consists of all elements of the form
\begin{equation}
\label{eq:P 2}
\sum_{i=1}^m a^{(i)}_{n_i}\cdot P_i\text{, as we vary }n_i\in\N,
\end{equation}
for some given elements $P_1,\dots, P_m\in\Gamma_0$, 
where $\{a^{(i)}_n\}_{n\in\mathbb{N}}$ are $S$-arithmetic sequences (for $1\le i\le m$). Then for each $i\in\{1,\dots, m\}$, we write $P_i$ as
\begin{equation}
\label{eq:P2}
P_i:=\sum_{j=1}^r c_{i,j}Q_j\text{ for some }c_{i,j}\in\Z.
\end{equation}
Therefore, a point in $\mathcal{F}=\mathcal{S}+H$ is of the form
\begin{equation}
\label{eq:generic point 0}
\sum_{k=1}^s y_k\cdot R_k + \sum_{i=1}^m a^{(i)}_{n_i}\cdot P_i,
\end{equation}
for some arbitrary integers $y_k$ and arbitrary positive integers $n_i$. Hence, using equations~\eqref{eq:R}~and~\eqref{eq:P2}, we write any point in $\mathcal{F}$ as:
\begin{equation}
\label{eq:generic point}
\sum_{j=1}^r \left(\sum_{k=1}^s b_{k,j}y_k + \sum_{i=1}^m c_{i,j}a^{(i)}_{n_i}\right)\cdot Q_j, 
\end{equation}
for some arbitrary $y_k\in\Z$ and some arbitrary $n_i\in\mathbb{N}$.

Now, following \cite[Definition~3.5]{TAMS}, we define ${\bf C}$-subsets, ${\bf L}$-subsets and ${\bf CL}$-sets of $\Z^r$; their notation comes from \emph{congruence} equation (for ${\bf C}$-subset), respctively \emph{linear} equation (for ${\bf L}$-subset).
\begin{defn}
\label{def:C-L}
A ${\bf C}$-subset of $\Z^m$ is a set ${\bf C}(d_1, \dots, d_r,D)$, where $d_1,\dots, d_r,D\in\Z$ (with $D\ne 0$), containing all solutions $(x_1,\dots, x_r)\in\Z^r$ of the congruence equation $\sum_{i=1}^r d_ix_i\equiv 0\pmod{D}$. 

An ${\bf L}$-subset of $\Z^r$ is a set ${\bf L}(d_1,\dots, d_r)$, where $d_1,\dots, d_r\in\Z$, containing all
solutions $(x_1, \dots, x_r)\in\Z^r$ of the linear equation $\sum_{i=1}^r d_ix_i=0$.

A ${\bf CL}$-subset of $\Z^r$ is either a ${\bf C}$-subset or an ${\bf L}$-subset of $\Z^r$.
\end{defn}

\begin{rem}
\label{eq:compare}
We note that our ${\bf C}$-subsets, ${\bf L}$-subsets and ${\bf CL}$-subsets are slightly simpler than the ones defined in \cite{TAMS} because we can absorb any coset of a subgroup in our definition of $S$-arithmetic groupless subsets of $\Gamma_0$ since the constant sequence is a linear recurrence sequence with unique characteristic root equal to $1$ (see also Lemma~\ref{lem:still_simple}~and~Remark~\ref{rem:repeatedly}). 
\end{rem}

As proven in \cite[Subclaim~3.6]{TAMS}, there exist finitely many ${\bf C}$-subsets $C_i$ of $\Z^r$ (with the index $i$ varying in some finite set $I$)   and there exist finitely many ${\bf L}$-subsets $L_j$ of $\Z^r$ (with the index $j$ varying in some finite set $J$) such that for any $(x_1,\dots, x_r)\in\Z^r$, we have 
\begin{equation}
\label{eq:generic C-L}
\sum_{i=1}^r x_iQ_i\in\Gamma
\end{equation}
if and only if 
\begin{equation}
\label{eq:generic C-L 2}
(x_1,\dots, x_r)\in \left(\bigcap_{i\in I}C_i\right)\bigcap \left(\bigcap_{j\in J}L_j\right). 
\end{equation}
Combining equations~\eqref{eq:generic point},~\eqref{eq:generic C-L}~and~\eqref{eq:generic C-L 2},  our problem reduces to analyzing for which integers $y_1,\dots, y_s$ and for which positive inttegers $n_1,\dots, n_m$, we have that 
\begin{equation}
\label{eq:general_equ}
\left(\sum_{k=1}^s b_{k,j}y_k + \sum_{i=1}^m c_{i,j}a^{(i)}_{n_i}\right)_{1\le j\le r}\in  \left(\bigcap_{i\in I}C_i\right)\bigcap \left(\bigcap_{j\in J}L_j\right).
\end{equation}
We analyze the conditions imposed on the tuples $(y_1,\dots,y_s,n_1,\dots, n_m)$ so that the left-hand side from equation~\eqref{eq:general_equ} belongs to some $C_i$, respectively some $L_j$. We split our analysis for these two cases over the next two Sections~\ref{subsec:C}~and~\ref{subsec:L}; there are significant differences between these two cases, one of them being that the case of ${\bf C}$-subsets is easier than the case of ${\bf L}$-subsets and it can be treated one congruence equation at a time.


\subsection{The case of ${\bf C}$-subsets of $\Z^r$}
\label{subsec:C}

In this Section~\ref{subsec:C}, we analyze the condition that the left-hand side of equation~\eqref{eq:general_equ} belongs to some given ${\bf C}$-subset $C\subseteq \Z^r$. Hence, for given integers $d_1,\dots, d_r$ and nonzero integer $D$, we analyze the equation
\begin{equation}
\label{eq:general_congruence_equ}
\sum_{j=1}^r d_j\cdot \left(\sum_{k=1}^s b_{k,j}y_k + \sum_{i=1}^m c_{i,j}a^{(i)}_{n_i}\right)\equiv 0\pmod{D}.
\end{equation}
We note that in equation~\eqref{eq:general_congruence_equ}, the \emph{unknowns} are the $y_k$'s and the $n_i$'s, while $D$, the $d_j$'s, the $b_{k,j}$'s and the $c_{i,j}$'s are given integers. 
Since any linear recurrence sequence of integers (such as each $\{a^{(i)}_n\}_{n\ge 1}$)  is preperiodic modulo any given integer (such as $D$), we obtain that there exist finitely many tuples of non-negative integers 
$$(u_1,v_1,u_2,v_2,\cdots, u_{s+m},v_{s+m})$$
such that the set of tuples $(y_1,\dots, y_s,n_1,\dots, n_m)$ satisfying equation~\eqref{eq:general_congruence_equ} is a finite union of sets of the form
\begin{equation}
\label{eq:u-v}
\left\{(u_1y_1+v_1,\cdots, u_sy_s+v_s,u_{s+1}n_1+v_{s+1},\cdots, u_{s+m}n_m+ v_{s+m})\colon \text{for }y_i\in\Z\text{ and }n_i\in\N\right\}.
\end{equation}
As before (see Lemma~\ref{lem:still_simple}~(i)), replacing each linear recurrence sequence $\{a^{(i)}_n\}_{n\in\mathbb{N}}$ by $\{a^{(i)}_{u_{s+i}n+v_{s+i}}\}_{n\in\mathbb{N}}$ (for $i=1,\dots, m$) leads to another $m$ linear recurrence sequences with simple characteristic roots that all live in the set $S$. Furthermore, replacing each $y_k$ by $u_ky_k+v_k$ leads to replacing the subgroup $H$ with a coset $\delta_1+H_1$ of a subgroup $H_1\subseteq H$. Then, once again using Remark~\ref{rem:repeatedly}, we conclude  that replacing each $y_k$ by $u_ky_k+v_k$ (for $1\le k\le s$) and replacing each $n_i$ by $u_{s+i}n_i+v_{s+i}$ (for $1\le i\le m$) leads to replacing the $S$-arithmetic subset $\mathcal{F}=\mathcal{S}+H$ by another $S$-arithmetic set $\mathcal{S}_1+H_1$, where $\mathcal{S}_1$ is the $S$-arithmetic groupless set 
$$\mathcal{S}_1:=\left\{\sum_{k=1}^s v_kR_k+\sum_{i=1}^m a^{(i)}_{u_{s+i}n_i+v_{s+i}}\cdot P_i\colon n_i\in\mathbb{N}\right\},$$
while $H_1$ is the subgroup of $H$ spanned by $u_1R_1,\cdots, u_kR_k$. 
Therefore, from now on, we may assume that in the right-hand side of the equation~\eqref{eq:general_equ} we only have ${\bf L}$-subsets $L_j$.


\subsection{The case of ${\bf L}$-subsets of $\Z^r$} 
\label{subsec:L}

So, with the notation as in equation~\eqref{eq:general_equ}, letting $|J|=u$, we need to find $(y_1,\dots, y_s,n_1,\dots, n_m)\in \Z^s\times \mathbb{N}^m$ such  that
\begin{equation}
\label{eq:general_equ 2}
\left(\sum_{k=1}^s b_{k,j}y_k + \sum_{i=1}^m c_{i,j}a^{(i)}_{n_i}\right)_{1\le j\le r}\in  \bigcap_{1\le h\le u}L_h,
\end{equation}
where for each $h\in\{1,\dots, u\}$, the subset $L_h\subseteq \Z^r$ is cut out  by the linear equation
\begin{equation}
\label{eq:linear_equation}
\sum_{j=1}^r d_{h,j}x_j=0,
\end{equation}
for some given integers $d_{h,j}$. Combining equation~\eqref{eq:general_equ 2} with the linear equations~\eqref{eq:linear_equation} (for $1\le j\le u$), we obtain a system of $u$ linear equations:
\begin{equation}
\label{eq:linear_equation 2}
\sum_{k=1}^s \left(\sum_{j=1}^r b_{k,j}d_{h,j}\right)\cdot y_k=- \sum_{i=1}^m \left( \sum_{j=1}^r d_{h,j}c_{i,j}\right)\cdot a^{(i)}_{n_i}\text{ for }1\le h\le u.
\end{equation}
For each $h=1,\dots, u$ and for each $k=1,\dots, s$, we let 
$$e_{h,k}:=\sum_{j=1}^r b_{k,j}d_{h,j};$$
also, for each $h=1,\dots, u$ and for each $i=1,\dots, m$, we let
$$f_{h,i}:=-\sum_{j=1}^r d_{h,j}c_{i,j}.$$
Clearly, $e_{h,k}\in\Z$ (for each $1\le h\le u$ and $1\le k\le s$) and $f_{h,i}\in\Z$ (for each $1\le h\le u$ and $1\le i\le m$); also, we recall that each $a^{(i)}_{n_i}\in\Z$ for $1\le i\le m$ and each $n_i\in\mathbb{N}$. So, we have a linear system of $u$ equations with unknowns $y_1,\dots, y_s$:
\begin{equation}
\label{eq:linear_equation 3}
\sum_{k=1}^s e_{h,k}y_k =  \sum_{i=1}^m f_{h,i}a^{(i)}_{n_i}\text{ for }1\le h\le u.
\end{equation}
In order for the system~\eqref{eq:linear_equation 3} be solvable for some $y_1,\dots, y_s\in\Q$,  there are finitely many linear relations to be satisfied by the right-hand side terms from~\eqref{eq:linear_equation 3}; since each $e_{h,k}$ is an integer, these linear relations will have integer coefficients as well. Hence, there are finitely many equations, say $w$ equations (for some $w\ge 0$), and there are some given integers $g_{\ell,h}$ with $1\le \ell\le w$ and $1\le h\le u$ such that:
\begin{equation}
\label{eq:linear_equation 4}
\sum_{h=1}^u g_{\ell,h}\cdot \sum_{i=1}^m f_{h,i}a^{(i)}_{n_i}=0,
\end{equation}
which need to be satisfied (for $1\le \ell\le w$) by the positive integers $n_i$ in order to find a solution $(y_1,\dots, y_s)$ (even over the rationals) for the system~\eqref{eq:linear_equation 3}. Letting 
$$z_{\ell,i}:=\sum_{h=1}^u g_{\ell,h}\cdot f_{h,i}\text{ for each }\ell\in\{1,\dots, w\}\text{ and for each }i\in\{1,\dots, m\},$$
then equations~\eqref{eq:linear_equation 4} translate to equations:
\begin{equation}
\label{eq:linear_equation 5}
\sum_{i=1}^m z_{\ell,i}\cdot a^{(i)}_{n_i}=0\text{ for }\ell=1,\dots, w.
\end{equation}
Since each sequence $\{a^{(i)}_n\}_{n\in\mathbb{N}}$ is a linear recurrence sequence with simple characteristic roots $\tilde{r}_{i,1},\dots, \tilde{r}_{i,m_i}\in S$ (for some $m_i\in\mathbb{N}$), then the equations~\eqref{eq:linear_equation 5} translate into equations:
\begin{equation}
\label{eq:linear_equation 50}
\sum_{i=1}^m\sum_{j=1}^{m_i} \tilde{z}_{\ell,i,j} \cdot \tilde{r}_{i,j}^{n_i}=0\text{ for }\ell=1,\dots, w,
\end{equation}
for some constants $\tilde{z}_{\ell,i,j}\in\C$.  
The famous theorem of Laurent \cite{Laurent} (which solves the classical Mordell-Lang conjecture for algebraic tori) yields that the set of tuples $(n_1,\dots, n_m)\in\mathbb{N}^m$ satisfying the equations~\eqref{eq:linear_equation 50} is a union  of finitely many sets $\tilde{S}_k$ of the following form. Each set $\tilde{S}_k$ consists of all tuples $(n_1,\dots, n_m)\in\mathbb{N}^m$ satisfying finitely many equations of the form:
\begin{equation}
\label{eq:M-L-tori 1}
n_j=\tilde{n}_{0,j}\text{ for some given }\tilde{n}_{0,j}\in\mathbb{N},
\end{equation}
or 
\begin{equation}
\label{eq:M-L-tori 2}
n_j\equiv \tilde{n}_{0,j}\pmod{\tilde{N}_{0,j}}\text{ for some given }\tilde{n}_{0,j},\tilde{N}_{0,j}\in\mathbb{N},
\end{equation}
or
\begin{equation}
\label{eq:M-L-tori 3}
\tilde{n}_{0,j}\cdot n_j=\tilde{n}_{0,j_1}\cdot n_{j_1}\text{ for some given }j\ne j_1\text{ and }\tilde{n}_{0,j},\tilde{n}_{0,j_1}\in\mathbb{N}.
\end{equation}
\begin{rem}
\label{rem:Laurent}
This is the only part of our proof which requires that in the definition of $S$-arithmetic sets the corresponding linear recurrence sequences have distinct characteristic roots. Indeed, otherwise, the problem becomes \emph{very difficult}. For example, consider the case when $\Gamma_0=\Z^2$, $\Gamma=\Z\times \{0\}$ and $\mathcal{F}=\mathcal{S}$ is the set of all elements of $\Z^2$ of the form:
\begin{equation}
\label{int:non-Laurent}
n_0^2\cdot (1,1) + \sum_{i=1}^m 2^{n_i}\cdot (1,-1)\text{ for arbitrary }n_0,n_1,\dots, n_m\in\mathbb{N}. 
\end{equation}
The set \eqref{int:non-Laurent} corresponds to the linear recurrence sequences $\{n_0^2\}_{n_0\ge 1}$ along with $\{2^{n_i}\}_{n_i\ge 1}$ (for $1\le i\le m$); their characteristic roots live in the powers-closed set $\{2^s\colon s\ge 0\}$, \emph{but} the first of these linear recurrence sequences has $1$ as a \emph{repeated characteristic root}. Then analyzing the counterpart of Theorem~\ref{thm:G} for this example leads to finding all solutions $(n_0,n_1,\dots, n_m)\in\N^{m+1}$ for the polynomial-exponential equation:
\begin{equation}
\label{eq:poly-expo}
n_0^2=\sum_{i=1}^m 2^{n_i}.
\end{equation}
The equation~\eqref{eq:poly-expo} is a \emph{very deep} Diophantine question, well-beyond the reach of the current known methods; for more details, see \cite{CGSZ}.  
\end{rem}

Thus, it suffices to prove that for each such set $\tilde{S}:=\tilde{S}_k\subseteq \mathbb{N}^m$ (satisfying finitely many equations of the form~\eqref{eq:M-L-tori 1},~\eqref{eq:M-L-tori 2}~and~\eqref{eq:M-L-tori 3}), the corresponding set of points~\eqref{eq:generic point} lying in $\mathcal{F}\cap\Gamma$ is an $S$-arithmetic subset in $\Gamma$. 

On the other hand, due to the simple form of the equations~\eqref{eq:M-L-tori 1},~\eqref{eq:M-L-tori 2}~and~\eqref{eq:M-L-tori 3}, the set 
\begin{equation}
\label{eq:still_groupless}
\left\{\sum_{i=1}^m a^{(i)}_{n_i}\cdot P_i\colon (n_1,\dots, n_m)\in\tilde{S}\right\}
\end{equation}
is still an $S$-arithmetic groupless set (see Lemma~\ref{lem:still_simple}~(i)).  Therefore, at the expense of replacing the original $S$-arithmetic groupless set $\mathcal{U}$ by the set from~\eqref{eq:still_groupless}, we may assume that the linear system of equations~\eqref{eq:linear_equation 3} is solvable (in $\Q$) for each $(n_1,\dots,n_m)\in\mathbb{N}^m$. 

So, at the expense of re-shuffling our variables $(y_1,\dots,y_s)$ (which amounts to re-labelling the points $R_1,\dots, R_s$), we may assume that $y_1,\dots, y_t$ (with $t\le s$) are free variables for the linear system~\eqref{eq:linear_equation 3} and the general solution in $\Q$ to the system~\eqref{eq:linear_equation 3} consists of tuples $(y_1,\dots, y_s)$ with the property that $y_1,\dots, y_t$ are arbitrary rational numbers, while for $j=1,\dots,s-t$, we have
\begin{equation}
\label{eq:linear_equation 12}
y_{t+j}=\sum_{k=1}^t \tilde{a}_{j,k}\cdot y_k + \sum_{\substack{1\le i\le m\\1\le h\le u}} \tilde{b}_{j,h}\cdot f_{h,i}\cdot a^{(i)}_{n_i},
\end{equation} 
for some given constants $\tilde{a}_{j,k},\tilde{b}_{j,h}\in\Q$, while the $n_i$'s from~\eqref{eq:linear_equation 12} are arbitrary positive integers. For each $j=1,\dots, s-t$ and each $i=1,\dots, m$, we let 
$$\tilde{c}_{j,i}:=\sum_{h=1}^u \tilde{b}_{j,h}f_{h,i};$$
so, we can re-write equation~\eqref{eq:linear_equation 12} as follows:
\begin{equation}
\label{eq:linear_equation 13}
y_{t+j}=\sum_{k=1}^t \tilde{a}_{j,k}\cdot y_k + \sum_{i=1}^m \tilde{c}_{j,i}\cdot a^{(i)}_{n_i}.
\end{equation}


\subsection{Conclusion for our proof of Theorem~\ref{thm:G}}
\label{subsec:conclusion}

Knowing that the constants  $\tilde{a}_{j,k}$ and $\tilde{c}_{j,i}$ from equation~\eqref{eq:linear_equation 13} are rational numbers, coupled with the fact that each linear recurrence sequence of integeres is preperiodic with respect to any given moduli, then $y_{t+j}\in\Z$ for each $j=1,\dots, s-t$ in  equation~\eqref{eq:linear_equation 13} yields that the corresponding set of all tuples  $(y_1,\dots, y_t,n_1,\dots,n_{m})\in  \Z^t\times \mathbb{N}^{m}$ consists of finitely many sets of the form
\begin{equation}
\label{eq:arith_progr 111}
\left\{\left(u_1y_1+v_1,\cdots, u_ty_t+v_t,u_{t+1}n_1+v_{t+1},\cdots , u_{t+m}n_{m}+v_{t+m}\right)\colon (y_1,\dots, y_t,n_1,\dots, n_{m})\in\Z^t\times \mathbb{N}^{m}\right\}
\end{equation}
for some given integers $u_j,v_j$ for $1\le j\le t+m$. So, we consider now a given choice of integers $u_j$ and $v_j$ as in equation~\eqref{eq:arith_progr 111} so that for each  $(y_1,\dots, y_t,n_1,\dots,n_{m})\in  \Z^t\times \mathbb{N}^{m}$, we have that $y_{t+1},\dots, y_s$ given as in~\eqref{eq:linear_equation 13} provide a solution over the integers to the system~\eqref{eq:linear_equation 3}.

Next, we show that the set of points in $\mathcal{F}$ corresponding to a subset $\tilde{E}\subseteq \Z^t\times \mathbb{N}^{m}$ of the form~\eqref{eq:arith_progr 111} is an $S$-arithmetic set. Indeed, we note first that replacing each $n_i$ (for $1\le i\le m$) by $u_{t+i}\cdot n_i+v_{t+i}$ yields that the corresponding points in $\mathcal{U}$: 
$$\sum_{i=1}^m a^{(i)}_{u_{t+i}n_i+v_{t+i}}\cdot P_i\text{ (as we vary $(n_1,\dots, n_m)\in\mathbb{N}^m$)}$$
still form an $S$-arithmetic groupless set (once again, see Lemma~\ref{lem:still_simple}~(i)). Second, we claim that the set of all points in $H$ of the form:
\begin{equation}
\label{eq:arith_set}
\sum_{k=1}^t (u_ky_k+v_k)\cdot R_k+ \sum_{j=1}^{s-t} y_{t+j}\cdot R_{t+j},
\end{equation}
with $y_{t+1},\dots, y_s$ given as in equation~\eqref{eq:linear_equation 13} (where each $y_i$ is replaced by $u_iy_i+v_i$ for $i=1,\dots, t$) is actually an $S$-arithmetic set. Indeed, the set from~\eqref{eq:arith_set} can be re-written as the sum of an $S$-arithmetic groupless subset (see also Lemma~\ref{lem:still_simple} and Remark~\ref{rem:repeatedly}): 
\begin{equation}
\label{eq:still_groupless 11}
\sum_{k=1}^t v_k\cdot R_k + \sum_{j=1}^{s-t}\left(\sum_{k=1}^t \tilde{a}_{j,k}v_k + \sum_{i=1}^{m}  \tilde{c}_{j,i}a^{(i)}_{u_{t+i}n_i+v_{t+i}}\right)\cdot R_{t+j}
\end{equation}
(as we vary $n_1,\dots, n_m$ freely in $\mathbb{N}$) 
with the subgroup of $H$ consisting of all points:
\begin{equation}
\label{eq:still_set 11}
\sum_{k=1}^t  u_ky_k\cdot R_k + \sum_{j=1}^{s-t}\left(\sum_{k=1}^t \tilde{a}_{j,k}u_ky_k\right)\cdot R_{t+j} ,
\end{equation}
as we vary $y_1,\dots, y_t$ freely in $\Z$. 

This concludes our proof of Theorem~\ref{thm:G}. 



\section{Proof of Theorem~\ref{thm:main}}
\label{sec:proofs}

Theorem~\ref{thm:main} follows now as an easy corollary of Theorem~\ref{thm:G}. 

So, we let $G$, $K$, $X$, $\Gamma$, $\Fq$, $F$ and $S_F$ be as in Theorem~\ref{thm:main} (see also Notations~\ref{not:semi}~and~\ref{def:still F-set}). We let $\tilde{\Gamma}$ be the finitely generated $\Z[F]$-submodule of $G(K)$ spanned by $\Gamma$; in particular, $\tilde{\Gamma}$ is also a finitely generated subgroup of $G(K)$. Using Theorem~\ref{thm:R-T}, the intersection $X(K)\cap \tilde{\Gamma}$ is a finite union of $F$-sets in $\tilde{\Gamma}$ (based in $G(K)$). Furthermore, according to Remark~\ref{rem:still F-set}, we have that each $F$-set in $\tilde{\Gamma}$ is also an $S_F$-arithmetic set in $\tilde{\Gamma}$ (based in $G(K)$). Finally, using Theorem~\ref{thm:G} (with $\Gamma_0:=G(K)$) coupled with the fact that 
$$X(K)\cap\Gamma=\left(X(K)\cap\tilde{\Gamma}\right)\cap \Gamma,$$
we conclude that $X(K)\cap\Gamma$ is a finite union of $S_F$-arithmetic sets in $\Gamma$ (based in $G(K)$), as desired.


\end{document}